\documentclass{tsp-short}
\usepackage{amsfonts}
\usepackage{amssymb}
\usepackage{newlfont}
\usepackage{amsthm}
\usepackage{euscript}
\usepackage{amsmath}
\usepackage{graphicx}
\usepackage{dsfont}
\usepackage{enumerate}

\theoremstyle{definition}
\newtheorem{Definition}{Definition}
\newtheorem{Lemma}{Lemma}
\newtheorem{Theorem}{Theorem}

\theoremstyle{Condition}

\newtheorem{Condition}{Condition}

\theoremstyle{remark}

\newtheorem{Remark}{Remark}

\newtheorem{Example}{Example}

\begin{document}

\title[On the strong uniqueness of a solution to a singular SDE]{On the strong uniqueness of a solution to singular stochastic differential equations}

\author{Olga V. Aryasova}
\address{Institute of Geophysics, National Academy of Sciences of Ukraine,
Palladin pr. 32, 03680, Kiev-142, Ukraine}
\email{oaryasova@mail.ru}

\author{Andrey Yu. Pilipenko}
\footnote{This work was partially supported by the State fund for fundamental
researches of Ukraine and the Russian foundation for basic research
under project F40.1/023.}
\address{Institute of Mathematics,  National Academy of Sciences of
Ukraine, Tereshchenkivska str. 3, 01601, Kiev, Ukraine}
\email{apilip@imath.kiev.ua}

 \dedicatory{}


\begin{abstract} We prove the existence and uniqueness of a strong solution  for an SDE on a semi-axis with singularities at the point $0$. The result obtained yields, for example, the strong uniqueness of non-negative solutions to SDEs governing Bessel processes.
\end{abstract}
\maketitle

\section*{Introduction}
We consider a stochastic process the state space of which is a non-negative semi-axis. Assume that up to the first hitting time of zero the  process $(x(t))_{t\geq0}$ satisfies an SDE
$$
x(t)=x_0+\int_0^t a(x(s))ds+\int_0^t \sigma(x(s))dw(s),
$$
where $x_0\geq0, a,\sigma$ are supposed to be locally Lipshitz
continuous on $(0,\infty)$, $(w(t))_{t\geq0}$ is a Wiener process.
Possible singularities of the coefficients generate different types of behavior of the process in a neighborhood of zero. As a consequence, the integral representation of $(x(t))_{t\geq0}$ may acquire various forms.

As an example let us consider  the following SDE
\begin{equation}\label{bessel}
\rho(t)=\rho(0)+w(t)+\frac{\beta-1}{2}\int_0^t\frac{1}{\rho(s)}ds, \ \rho(0)\geq0.
\end{equation}
       It is known that $\beta$-dimensional Bessel process with $\beta>1$ is a unique non-negative strong solution to (\ref{bessel}) (cf. \cite{Cherny00}). Note that this equation
possesses no additional terms. Otherwise,
an additional summand can be represented by the local time $(l(t))_{t\geq0}$ of unknown process $(x(t))_{t\geq0}$ at the point $0$ like in  Skorokhod equation
\begin{equation}\label{Skorokhod1}
x(t)={{x_0}}+\int_0^t a(x(s))ds+\int_0^t b(x(s))dw(s)+l(t),
\end{equation}
or by principal values of some functionals
of the unknown process like in the following representation for a
$\beta$-dimensional Bessel process with $\beta\in(0,1)$ (cf. \cite{Revuz+99}, Ch. XI)
\begin{equation}\label{bessel1}
\rho(t)=\rho(0)+w(t)+\frac{\beta-1}{2}k(t), \ \rho(0)\geq0.
\end{equation}
  Here $k(t)=V.P.\int_0^t\rho^{-1}(s)ds$ which, by definition, is equal to $\int_0^\infty a^{\beta-2}(L_a^\rho(t)-L_0^\rho(t))da$, $L_a^\rho(t)$ ia a local time of the process $\rho(t)$ at the point $a$.

It seems improbable to describe all possible forms of integral representations. Let $f$ be a twice continuously
differentiable function on $[0,\infty)$ which is a constant in a neighborhood of zero. Applying It$\hat{o}$  formula (additional tricks are needed in some cases) we see that the stochastic differential of the function $f(x(t))$ has identical form for solutions of equations (\ref{bessel})-(\ref{bessel1}). Namely,
\begin{multline}\label{main_eq}
f(x(t))=f(x_0)+\int_0^t\left(a(x(s))f'(x(s))+\frac{1}{2}\sigma^2(x(s))f''(x(s))\right)ds\\+
\int_0^t \sigma(x(s))f'(x(s))dw(s).
 \end{multline}
The differential has no singularities. Intuitively, the singularities at the point $0$ are killed by zero derivatives of the function $f$. We use this fact
to formulate the problem as an analogue of a
martingale problem
(see Section 1). The main result of this paper is as follows: the
existence of a weak non-negative solution for equation
(\ref{main_eq}) spending zero time at the point $0$ implies the
existence and uniqueness of a non-negative strong solution
spending zero time at $0$. The pathwise uniqueness is obtained by
method of Le Gall \cite{LeGall83} based on the fact that
the maximum of two solution also solves the equation. The
formulation of a martingale problem involving a class of functions
that are constant in a neighborhood of possible singularities was
used by many authors (see, for example, \cite{Ramanan06},
\cite{Varadhan+85}).

The notations and definitions used are collected in Section 1. We
prove the main Theorem in Section 2. In Section 3 some examples are
represented.

\section{Notations and definitions}
Let $a,\sigma$ be real-valued Borel-measurable functions defined
on $[0,\infty)$. From now on we assume that the following condition is valid
\begin{Condition} Suppose that the functions $a$ and $\sigma$ are locally Lipschitz continuous on $(0,\infty)$, i.e. for each $\varepsilon>0$ there exist
constants $C_{\varepsilon}>0$ such that for all $\{x,y\}\subset[\varepsilon,\infty)$
$$
|a(x)-a(y)|+|\sigma(x)-\sigma(y)|\leq C_{\varepsilon}|x-y|.
$$
\end{Condition}

The set of continuous functions
$x:[0,\infty)\rightarrow[0,\infty)$ is denoted by
$C^+([0,+\infty))$. Let $\mathfrak{G}_{t}\equiv\sigma\{x(s):0\leq
s\leq t, \ x\in C^+([0,+\infty)) \}$, and
$\mathfrak{G}\equiv\sigma\{x(s):0\leq s< \infty, \  x\in
C^+([0,+\infty)) \}$ be $\sigma$-algebras on  $C^+([0,+\infty))$.
The set of real-valued functions which are twice continuously
differentiable on $[0,\infty)$   and constant in a neighborhood of zero is denoted by
$C^2_c([0,+\infty)).$ Given a probability measure $P$ on $(C^+([0,+\infty)),\mathfrak{G})$, the family of continuous, square integrable
local $\mathfrak{G}_{t}$-martingales is denoted by $\mathcal{M}_2^{c,loc}(P).$

\begin{Definition} \label{mart_prob} Given $x_0 \geq0$, {\it a solution to the martingale problem  $M(a,\sigma,x_0)$}  is a probability measure $P_{x_0}$ on $(C^+([0,+\infty)),\mathfrak{G})$ such that
\begin{enumerate}[(i)]
\item $P_{x_0}(x(0)=x_0)=1.$
\item For each $f\in C^2_c([0,+\infty))$,
$$
Y_f(t)=f(x(t))-f(x_0)-\int_0^t\left[a(x(s))f'(x(s))+\frac{1}{2}\sigma^2(x(s))f''(x(s))\right]ds\in\mathcal{M}_2^{c,loc}(P_{x_0}).
$$
\item
$$
E^{P_{x_0}}\int_0^\infty\mathds{1}_{\{0\}}(x(s))ds=0.
$$
\end{enumerate}
\end{Definition}
\begin{Definition}\label{well-posed} The martingale problem is {\it well-posed} if for each $x_0\geq0$ there is exactly one solution
 to the martingale problem starting from $x_0$.
\end{Definition}

\begin{Definition}\label{Def_Weak_Sol} Given $x_0\geq0$, let a pair $\left(x(t),w(t)\right)_{t\geq0}$ of continuous adapted
processes on a filtered probability space $(\Omega,\mathfrak F,
\left( {\mathfrak F}_t\right), P)$  be such that
\begin{enumerate}[(i)]
\item $\left(w(t)\right)_{t\geq0}$ is a standard $({\mathfrak
    F}_t)$-Brownian motion,
    \item the process $(x(t))_{t\geq0}$ takes values on
        $[0,\infty)$,
    \item for each $t\geq 0$, and $f\in C^2_c([0,\infty))$,
        the equality
    \begin{multline} \label{f_martingale}
    f(x(t))=f(x_0)+\int_0^t\left(a(x(s))f'(x(s))+\frac{1}{2}\sigma^2(x(s))f''(x(s))\right)ds\\+
    \int_0^t \sigma(x(s))f'(x(s))dw(s)
    \end{multline}
    holds true $P$-a.s.

    Then the pair  $\left(x,w\right)$ is called {\it a weak
    solution to equation (\ref{f_martingale}) with initial
    condition $x_0$}.
\end{enumerate}
\end{Definition}
\begin{Remark}\label{integtal_existance}
Let $f\in C_c^2([0,\infty))$. Then there exists $\delta_f>0$ such that $f'=f''=0$ on $[0,\delta_f]$. This and Condition A ensure the existence of all the integrals on the right-hand side of (\ref{f_martingale}).
\end{Remark}
\begin{Remark}\label{equivalent}
It is not hard to verify that the existence of a weak solution
$(x(t),w(t))_{t\geq0}$ to equation (\ref{f_martingale}) on a
probability space $(\Omega,\mathfrak F, P)$ with initial condition
$x_0$ is equivalent to the  existence of  a probability measure
$\tilde P$ on some probability space
$(\tilde\Omega,\tilde{\mathfrak F}, \tilde P)$ satisfying
conditions (i) and (ii) of Definition \ref{mart_prob}. The process
$(x(t))_{t\geq0}$ induces the measure $\tilde P$ on
$(C^+([0,\infty)),\mathfrak{G}),$ namely  $\tilde P=Px^{-1}$.

For the proof see Appendix.
\end{Remark}

\begin{Definition}
{\it The  weak uniqueness} holds for equation (\ref{f_martingale})
if, for any two weak solutions $\left(x,w\right)$ and
$\left(\tilde x,\tilde w\right)$ (which may be defined on
different probability spaces) with a common initial value, i.e.
$x(0)=x_0$ $P$-a.s., $\tilde x(0)=x_0 \ \tilde P$-{a.s.}  the laws
of processes $x$ and $\tilde x$ coincide.\label{Def_Weak_Uniq}
\end{Definition}

\begin{Definition}
{\it The pathwise uniqueness} holds for equation
(\ref{f_martingale}), if for any two weak solutions
$\left(x,w\right)$ and $\left(\tilde x, w\right)$  on the same
probability space $(\Omega,\mathfrak F, P)$ with common Brownian
motion and common initial value, i.e. $x(0)=\tilde x(0)=x_0\ \
P$-a.s.,  the equality $x(t)=\tilde x(t), \ t\geq 0,$ fulfils
$P$-a.s.\label{Def_Path_Uniq}
\end{Definition}

Denote by  $\left(\overline{\mathfrak{F}}_t^w\right)$ the
filtration of $w$ completed with respect to $P$.
\begin{Definition} \label{Def_Str_Sol}
Given  a process $(w(t))_{t\geq0}$, and
$x_0\geq0$, we say that the process $(x(t))_{t\geq0}$ is {\it a
strong solution} to  equation (\ref{f_martingale}) with initial condition $x_0$  if
it is adapted to the filtration
$\left(\overline{\mathfrak{F}}_t^w\right)$ and  conditions (i)-(iii) of
Definition \ref{Def_Weak_Sol} hold.
\end{Definition}
\begin{Definition} \label{Def_Str_Unique}
{\it The strong uniqueness} holds for equation
(\ref{f_martingale}) if there exists a strong solution to equation
(\ref{f_martingale}) and the pathwise uniqueness is valid for
equation (\ref{f_martingale}).
\end{Definition}


\section {The main result}
Unfortunately, we are not able to write equation
(\ref{f_martingale}) for the process $(x(t))_{t\geq0}$ itself
because the function $f(x)=x$ does not belong to
$C_c^2([0,\infty))$. Instead, in the next Lemma we obtain an SDE
for the process $\zeta_\delta(x(t))=x(t)\vee\delta, \ t\geq0,$  which
will often be used in the sequel.
\begin{Lemma}\label{approxim}
{\it Given $\delta>0$, put $\zeta_\delta(x)=x\vee \delta, \ x\in[0,\infty).$
Suppose $(x(t))_{t\geq0}$ is a weak solution to equation
(\ref{f_martingale}). Then the equality
\begin{multline}\label{eta_delta_x}
\zeta_\delta(x(t))=\zeta_\delta(x(0))+\int_0^t a(x(s))\mathds{1}_{(\delta,+\infty)}(x(s))ds\\+
\int_0^t\sigma(x(s))\mathds{1}_{(\delta,+\infty)}(x(s))dw(s)+\frac{1}{2}L_\delta^{x}(t)
\end{multline}
is valid for all $t\geq0$.  Here  $(L_\delta^x(t))_{t\geq0}$ is a
local time of the process $(x(t))_{t\geq0}$ at the point $\delta$
defined by the formula}
\begin{equation}\label{local time}
L_\delta^x(t)=\lim_{\varepsilon\downarrow0}\frac{1}{\varepsilon}\int_0^t\mathds{1}_{[\delta, \delta+\varepsilon)}(x(s))ds, \ t\geq0.
\end{equation}
\end{Lemma}
\begin{proof} We make use of a standard approximation of
non-smooth function by smooth ones like the construction in the
proof of Tanaka's formula (see, for example, Theorem 4.1, Ch. 3 of
\cite{Ikeda+81}). For $a>0$, let us approximate the function
$\xi_a(x)=x\vee a$ by twice continuously differentiable functions.
Put
\begin{equation*}
\psi(x)=
\begin{cases}  C\exp{\left(\frac{1}{(x-1)^2-1}\right)}, & 0<x<2, \\
0, & \mbox{otherwise},
\end{cases}
\end{equation*}
where $C$ is a constant such that $\int_{-\infty}^{\infty}\psi(x)dx=1,$ and define, for $n\geq1$
$$
u_n(x)=\int_{-\infty}^x dy\int_{-\infty}^y n\psi(nz)dz.
$$
Then the function $u_n$ is twice continuously differentiable and
$u_n(x-a)+a\rightarrow \xi_a(x)$, as $n\to\infty$. Further,
$$
u_n'(x-a)\to\mathds{1}_{(a,\infty)}(x), \ n\to\infty,
$$
and
\begin{equation}\label{delta_func}
\int_{-\infty}^{\infty}u_n''(x-a)\varphi(x)dx\to\varphi(a),\ n\to \infty
\end{equation}
for any continuous and bounded function $\varphi$.

Let $\eta_\delta\in C_c^2([0,\infty))$  be a non-decreasing  and such
that $\eta_\delta(x)=x$ on $[\delta/2,\infty)$.  The process $\left(\eta_\delta(x(t))\right)_{t\geq0}$ can be represented in
the form (\ref{f_martingale}), and thus it is a semimartingale. By
It$\hat{o}$ formula we have
\begin{multline}\label{u_n}
u_n(\eta_\delta(x(t))-\delta)+\delta=u_n(\eta_\delta(x(0))-\delta)+\delta\\+
 \int_0^t \left[a(x(s))\eta'_\delta(x(s))+\frac{1}{2}\sigma^2(x(s))\eta''_\delta(x(s))\right]u_n'(\eta_\delta(x(s))-\delta)ds\\+
\int_0^t\sigma(x(s))\eta'_\delta(x(s))u_n'(\eta_\delta(x(s))-\delta)dw(s)\\+
\frac{1}{2}\int_0^t\sigma^2(x(s))(\eta'_\delta(x(s)))^2
u_n''(\eta_\delta(x(s))-\delta)ds.
\end{multline}
By occupation times formula (cf. \cite{Protter04}, Corollary 1, p.
216),  the last integral on the right-hand side of (\ref{u_n}) is
equal to
\begin{multline*}
\int_0^t u_n''(\eta_\delta(x(s))-\delta)d\langle \eta_\delta(x)\rangle(s)=\int_{-\infty}^{+\infty}u_n''(a-\delta)L_a^{\eta_\delta(x)}(t)da\to
L_\delta^{\eta_\delta(x)}(t), \\ \text{as } \  n\to\infty.
\end{multline*}
Here $L_\delta^{\eta_\delta(x)}(t)$ is a local time of the process
$(\eta_\delta(x(t)))_{t\geq0}$ at the point $\delta$.

Note that $\zeta_\delta(x)=x\vee\delta=\eta_\delta(x)\vee\delta, \ x\in[0,\infty)$. Passing to the
limit in (\ref{u_n}) as $n\to \infty$ and taking into account that
$\eta_\delta(x)=x, \ x>\delta/2,$  we arrive at the equation
(\ref{eta_delta_x}).
\end{proof}
The main result of the paper is the following theorem.
\begin{Theorem}\label{Theorem2}
{ \it Suppose $a,\sigma$ satisfy Condition A and $\sigma(x)\neq0, \
x\geq0$. If for each ${x_0}\geq0$ there exists a solution to the
martingale problem $M(a,\sigma,x_0)$, then for each $x_0\geq0$ there
exists a strong solution to equation (\ref{f_martingale}) with
initial condition $x_0$ spending zero time at the point $0$ and
the strong uniqueness holds in the class of solutions spending
zero time at $0$.}
\end{Theorem}

We split the proof of  Theorem into two steps. At the first one we
show that the existence of a solution to the martingale problem
provides well-posedness. At the second one  the pathwise
uniqueness is obtained from weak uniqueness. These two steps are
formulated as Lemmas in the following way.

\begin{Lemma}[weak uniqueness]\label{Prop well-posed} {\it Suppose $a,\sigma$ satisfy the conditions of Theorem \ref{Theorem2}. Let for each $\ {x_0}\geq0$,
there exists a solution to martingale problem $M(a,\sigma,{x_0})$.
Then the weak uniqueness holds for equation (\ref{f_martingale}).}
\end{Lemma}
\begin{proof} We would like to get a law of the process $(x(t))_{t\geq 0}$.
But we don't know an integral representation for $(x(t))_{t\geq 0}$ itself.
Instead, we consider the process $(x(t)\vee\delta)_{t\geq 0}$.
Applying a space transformation and change of time to the process $(x(t)\vee\delta)_{t\geq 0}$ we will see
that the law of the process obtained coincides
with that of the Wiener process with reflection at the point $\delta$.

Similarly to Theorem 12.2.5 of \cite{Strook+79} it can be shown
that the existence of solution to the martingale problem for each $x_0\geq0$ implies
the existence of a strong Markov, time homogeneous measurable Markov
family $\{\tilde P_{x_0}: \ {x_0}\in[0,\infty)\}$ such that for
each ${x_0}\in[0,\infty)$, $\tilde P_{x_0}$ is a solution to the
martingale problem starting from ${x_0}$. And by Theorem 12.2.4 of
\cite{Strook+79}
 to prove the uniqueness of a solution to the martingale problem it is sufficient to prove the uniqueness only for the family of strong Markov, time homogeneous
 solutions. If $\tilde P_{x_0}$   is such a solution starting from ${x_0}$, then according to Remark \ref{equivalent} there exists a
 pair  $(x,w)$  on some probability space $(\Omega,\mathcal{F}, P_{x_0})$ which is  a weak solution to equation (\ref{f_martingale})
  and $ P_{x_0} x^{-1}=\tilde P_{x_0}$. This yields that $(x(t))_{t\geq0}$ is a strong Markov and time homogeneous process.

Note that if the process $(x(t))_{t\geq0}$ does not hit zero the
assertion of Lemma is trivial. So from now on we suppose that
starting from ${x_0}$ the process  $(x(t))_{t\geq0}$ hits zero
$P$-a.s.

We follow the proof of Theorem 2.12 of \cite{Cherny+05}.
Denote
\begin{eqnarray*}
\rho(x)&=&\exp\left(\int_x^1 \frac{2a(y)}{\sigma^2(y)}dy\right), \ x\in(0,1],\\
s(x)&=&\begin{cases} \int_0^x\rho(y)dy\ & \mbox{if} \ \int_0^1\rho(y)dy<\infty,\\
                     -\int_x^1\rho(y)dy\ & \mbox{if} \ \int_0^1\rho(y)dy=\infty.
       \end{cases}
\end{eqnarray*}
Let $\zeta_\delta(x(t))=x(t)\vee\delta, \ t\geq0$. Then by Lemma 1
\begin{multline*}
\zeta_\delta(x(t))=\zeta_\delta(x(0))+\int_0^t a(x(s))\mathds{1}_{(\delta,+\infty)}(x(s))ds\\+
\int_0^t\sigma(x(s))\mathds{1}_{(\delta,+\infty)}(x(s))dw(s)+\frac{1}{2}L_\delta^{x}(t),
\ t\geq0.
\end{multline*}

Set $\Delta=s(\delta), y(t)=s(x(t)\vee\delta)=s(x)\vee\Delta, \ t\geq0$.
By It$\hat{o}$-Tanaka formula applied to the function $x\mapsto
s(x)\vee\Delta$, we have
$$
y(t)=y(0)+\int_0^t \rho(x(s))\sigma(x(s))\mathds{1}_{(\delta,+\infty)}(x(s))dM(s)+\frac{1}{2}\rho(\delta)L_\delta^x(t),
$$
where $M(s)=\int_0^t\sigma(x(s))dw(s).$
Applying It$\hat{o}$-Tanaka formula to the function $y\mapsto y\vee\Delta$, we get
\begin{multline*}
y(t)=y(t)\vee\Delta=y(0)+\int_0^t
\rho(s^{-1}(y(s)))\sigma(x(s))\mathds{1}_{(\Delta,+\infty)}(y(s))dM(s)+\frac{1}{2}\rho(\delta)L_\Delta^y(t)\\=
y(0)+N(t)+\frac{1}{2}\rho(\delta)L_{\Delta}^y(t),
\end{multline*}
where $N(t)=\int_0^t
\rho(s^{-1}(y(s)))\sigma(x(s))\mathds{1}_{(\Delta,+\infty)}(y(s))dM(s)$.

Consider $D_{t}=\int_0^t\mathds{1}_{(\Delta,+\infty)}(y(s))ds$.  Let us show that $D_{t}\to\infty$ as $t\to \infty$. Set
for $a,b>0$, $T_{a,b}=\inf\{t>0:x(t)=a\ \mbox{or}\ b\}$. Define
\begin{eqnarray*}
\mu_0&=&\inf\{t>0:x(t)=0\}, \mbox {and for }\  k=1,2,\dots,\\
\mu_k&=&\inf\{t>0: t\geq\mu_{k-1}+1,\ x(t)\geq2\delta \ \mbox{or}\  x(t)=0\}
\end{eqnarray*}
If the process can hit zero in finite time then for all $y\in[0,2\delta]$, $
P_y(T_{0,2\delta}<\infty)=1$ (cf. \cite{Cherny+05}, Section 2). Then
\begin{multline}\label{mu_1}
P_0(\mu_1<\infty)=P_0(x(1)>2\delta)+\int_{[0,2\delta]}P_y(T_{0,2\delta}<\infty)P_0(x(1)\in dy)\\=
P_0(x(1)>2\delta)+P_0(x(1)\in[0,2\delta])=1.
\end{multline}
Let
\begin{eqnarray*}
\tau_1&=&\inf\{t\geq0: x(t)=2\delta\}, \ \mbox{for}\ k=1,2,\dots,\\
\varkappa_k&=&\inf\{t>\tau_k: x(t)=\delta\}, \ \mbox {and for}\ k=2,3,\dots, \\
\tau_k&=&\inf\{t>\varkappa_{k-1}: x(t)=2\delta\}.
\end{eqnarray*}
Equality (\ref{mu_1}) yields $P_0(\tau_1<\infty)=1$. Indeed, note that $P_0(x(\mu_1)=0)=\alpha\in(0,1).$ Then, by strong Markov property
\begin{multline*}
P_0(\tau_1=\infty)\leq P_0(\tau_1\geq\mu_n)\leq P_0(\bigcap_{k=1}^n(x(\mu_k)=0))\\=
(P_0(x(\mu_1)=0))^n=\alpha^n\to0, \ \mbox{as} \ n\to\infty.
\end{multline*}
Thus $P_0(\tau_1<\infty)=1$. Let for $k=1,2,\dots,$ $\ \zeta_k=\varkappa_k-\tau_k$. Then $\{\zeta_k:k\geq1\}$ is a sequence of positive independent identically distributed random variables. Consequently, $D_\infty\geq\sum_{k=1}^n\zeta_k\to\infty,$ as $n\to\infty$. So $\lim_{t\to +\infty}D_t=+\infty$

Put
$$
\varphi_\delta(t)=\inf\{s\geq0:D(s)>t\},
$$
and
\begin{equation}\label{K(t)1}
U(t)=y(\varphi_\delta(t))=U(0)+N(\varphi_\delta(t))+\frac{1}{2}\rho(\delta)L_\Delta^y(t),\
t\geq0.
\end{equation}

It can be seen that the process $K(t)=N(\varphi_\delta(t))$ is a
martingale and
$$
\langle K\rangle(t)=\int_0^t\varkappa^2(U(s))ds, \ t\geq0,
$$
where
$
\varkappa(x)=\rho(s^{-1}(x))\sigma(s^{-1}(x)), \ x>0.
$
By It$\hat{o}$-Tanaka formula we have
\begin{equation}\label{K(t)2}
U(t)=U(0)+\int_0^t\mathds{1}_{(\Delta,+\infty)}(U(s))dK(s)+\frac{1}{2}\int_0^t\mathds{1}_{(\Delta,+\infty)}(U(s))dL_{\Delta}^y(\varphi_\delta(t))+\frac{1}{2}L_\Delta^U(t), \ t\geq0.
\end{equation}
Making use change of variables in Lebesgue-Stiltjes integrals and
taking into account that measure $dL_\Delta^y(\varphi_\delta(t))$ increases
only on the set $\{t\geq0:y(\varphi_\delta(t))=\Delta\}$, we arrive at the
equality
$$
\int_0^t\mathds{1}_{(\Delta,+\infty)}(U(s))dL_\Delta^{y}(\varphi_\delta(s))=\int_{\varphi_\delta(0)}^{\varphi_\delta(t)}\mathds{1}_{(\Delta,+\infty)}(y(s))dL_\Delta^{y}(s)=0.
$$
Comparing (\ref{K(t)1}) with (\ref{K(t)2}) we get from the
uniqueness of the semimartingale decomposition of $U$ that
$$
\int_0^t\mathds{1}_{(\Delta,+\infty)}(U(s))dK(s)=K(t),\ t\geq0,
$$
and
$$
U(t)=U(0)+K(t)+\frac{1}{2}L_\Delta^U(t),\ t\geq0.
$$
Consider
$$
A(t)=\int_0^t\varkappa^2(U(s))ds,
$$
and put
$$
A(\infty)=\lim_{t\rightarrow\infty}\varkappa^2(U(s))ds,
$$
$$
\tau(t)=\inf\{s\geq0:A(s)>t\},\ 0\leq t<A(\infty).
$$
Arguing as above we arrive at the equation
$$
V(t)=U(\tau(t))=V(0)+J(t)+\frac{1}{2}L_\Delta^V(t),\  0\leq t<A(\infty),
$$
where $J(t)=K(\tau(t)),\ 0\leq t <A(\infty)$, and $\langle
J\rangle(t)=\int_0^{\tau(t)}\varkappa^2(U(s))ds=t.$  By Theorem
7.2, Ch.2 of \cite{Ikeda+81} there exists a Brownian motion
$(w(t))_{t\geq0}$ (defined, possibly, on an enlarged probability
space) such that $J$ coincides with $w$  on $[0,A(\infty))$.
Skorokhod's lemma (cf. \cite{Revuz+99}, Ch.VI, Lemma 2.1) and
Lemma 2.3, Ch.VI of \cite{Revuz+99} allow us make the conclusion
that the process $V$ is a Brownian motion started at
$V(0)=s(x(0))\vee\Delta$, reflected at $\Delta$. Thus the measure
$P^\delta=Law(U(t):t\geq0)=Law(y(\varphi_\delta(t)): t\geq0)$ is determined
uniquely and does not depend on the choice of a solution $\tilde
P_{x_0}$. This entails that the law of the process
$x(\varphi_\delta(t))\vee\delta$ is uniquely defined. Note that item $(iii)$
of Definition \ref{mart_prob} provides that the process
$(x(t))_{t\geq0}$ spends zero time at the point 0 $P_{x_0}$-a.s.
Then for each $T>0$, $\varphi_\delta(t)\rightrightarrows t$ on $[0,T]$ and,
consequently, $x(\varphi_\delta(t))\rightrightarrows x(t)$, as
$\delta\downarrow0$ $P_{x_0}$-a.s. Therefore, $Law(x(t):t\geq0)$ is
defined uniquely and does not depend on the choice of the solution
$\tilde P_{x_0}$. Then according to Remark \ref{equivalent} the
weak uniqueness holds for equation (\ref{f_martingale}).
\end{proof}
\begin{Lemma}[pathwise uniqueness]\label{pathwise_uniqueness}
{\it Let the weak uniqueness hold for equation
(\ref{f_martingale}). Then the pathwise uniqueness holds true for
(\ref{f_martingale}).}
\end{Lemma}
\begin{proof}
Let $(x_1(t))_{t\geq0}, (x_2(t))_{t\geq0}$ be processes defined on the same probability space $(\Omega,\mathfrak F,
\left( {\mathfrak F}_t\right), P)$ and let each of them is a weak solution to equation (\ref{f_martingale}).
The idea of the proof is as follows. We will see that the process
$((x_1\vee x_2)(t))_{t\geq0}$ is also a weak solution to equation (\ref{f_martingale}).

By Theorem IV-68 of \cite{Protter04} we have
\begin{multline}\label{vee_1}
(\zeta_\delta(x_1)\vee \zeta_\delta(x_2)) \left(t\right)=\zeta_\delta(x_1(t))+(\zeta_\delta(x_2(t))-\zeta_\delta(x_1(t))^+\\=
\zeta_\delta({x_0})+\int_0^t\mathds{1}_{\zeta_\delta(x_2(s))-\zeta_\delta(x_1(s))>0}d\zeta_\delta(x_2(s))
+\int_0^t\mathds{1}_{\zeta_\delta(x_2(s))-\zeta_\delta(x_1(s))\leq0}d\zeta_\delta(x_1(s))\\+
\frac{1}{2}L_0^{\zeta_\delta(x_1)-\zeta_\delta(x_2)}(t),
\end{multline}
where $L_0^{\zeta_\delta(x_1)-\zeta_\delta(x_2)}(t)$ is a local
time of the process
$(\zeta_\delta({x}_1(t))-\zeta_\delta({x}_2(t)))_{t\geq0}$ at $0$.
Then
\begin{multline}\label{max_Ito}
\left(\zeta_\delta(x_1)\vee\zeta_\delta(x_2)\right)(t)=\zeta_\delta({x_0})+\int_0^t a((x_1 \vee x_2)(s))\mathds{1}_{(\delta,+\infty)}((x_1\vee x_2)(s))ds
\\+
\int_0^t\sigma((x_1\vee x_2)(s))
\mathds{1}_{(\delta,+\infty)}((x_1\vee x_2)(s))dw(s)\\+
\frac{1}{2}L_\delta^{x_1\vee x_2}(t)+\frac{1}{2}L_0^{\zeta_\delta(x_1)-\zeta_\delta(x_2)}(t).
\end{multline}
Consider the last summand in the right-hand side of (\ref{max_Ito}).

The properties of the local time (cf. Theorem 69,
\cite{Protter04}, p. 214) implies that
$L_0^{\zeta_\delta(x_1)-\zeta_\delta(x_2)}(\cdot)$ increases only on
$\{t:\zeta_{\delta}(x_1(t))=\zeta_{\delta}(x_2(t))\}$. We prove that
increases only on $\{t:\zeta_\delta(x_1(t))=\zeta_\delta(x_2(t))=\delta\}$.

Let $q\in[0,\infty)\cap \mathds{Q}$ be such that
$\zeta_\delta(x_1(q))>\delta, \zeta_\delta(x_2(q))>\delta$.Define
\begin{eqnarray*}
a_q& = &\sup \{t<q:(\zeta_\delta(x_1)\wedge\zeta_\delta(x_2))(t)=\delta\},\\
b_q& = &\inf \{t>q:(\zeta_\delta(x_1)\wedge\zeta_\delta(x_2))(t)=\delta\},
\end{eqnarray*}
and $I_q=(a_q,b_q).$ Suppose $\zeta_\delta(x_1)(q)=\zeta_\delta(x_2)(q)$.
Then by Theorem on homeomorphisms of flows (cf. Theorem V-46,
\cite{Protter04}) applied to equation (\ref{eta_delta_x}), we have
$\zeta_\delta(x_1(t))=\zeta_\delta(x_2(t)), \ t\in[q, b_q)$. On the other
hand, if there exists $r<q$ such that
$\zeta_\delta(x_1(r))\neq\zeta_\delta(x_2(r)),$ by the same theorem
$\zeta_\delta(x_1(q))\neq\zeta_\delta(x_2(q))$. Thus
$\zeta_\delta(x_1(q))=\zeta_\delta(x_2(q))$ implies
$\zeta_\delta(x_1(t))=\zeta_\delta(x_2(t)), t\in I_q,$ and
$(\zeta_\delta(x_1)\vee \zeta_\delta(x_2))(t)=\zeta_\delta(x_1(t)), t\in I_q$.
Comparison (\ref{max_Ito}) with (\ref{eta_delta_x}) permits the
conclusion that $L_0^{\zeta_\delta(x_1)-\zeta_\delta(x_2)}(t)=0, t\in
I_q$. In the case of $\zeta_\delta(x_1(q))\neq\zeta_\delta(x_2(q))$ we get
$\zeta_\delta(x_1(t))\neq\zeta_\delta(x_2(t)), \ t\in I_q$. So for every
$[\alpha,\beta]\in I_q$ there exists $\varepsilon_0>0$ such that
$|\zeta_\delta(x_1(t))-\zeta_\delta(x_2(t))|>\varepsilon_0, t\in[\alpha,\beta].$
Then for all $\varepsilon\in[0,\varepsilon_0),\
\mathds{1}_{[0,\varepsilon]}|\zeta_\delta(x_2(t))-\zeta_\delta(x_1(t))|=0,\ t\in
[\alpha,\beta].$ From Corollary 3 of \cite{Protter04}, p.225 we
obtain
\begin{multline*}
L_0^{\zeta_\delta(x_1)-\zeta_\delta(x_2)}(t)\\=
\lim_{\varepsilon\downarrow0}\frac{1}{\varepsilon}\int_0^t\mathds{1}_{[0,\varepsilon]}|\zeta_\delta(x_2(s))-\zeta_\delta(x_1(s))|d\langle\zeta_\delta(x_1)-\zeta_\delta(x_2)\rangle
(s)=0,\ t\in [\alpha,\beta].
\end{multline*}

Therefore, $L_0^{\zeta_\delta(x_1)-\zeta_\delta(x_2)}(\cdot)$ can increases
only on $\{t:\ \zeta_\delta(x_1(t))=\zeta_\delta(x_2(t))=\delta\}$, i.e. on $\{t:\
(\zeta_\delta(x_1)\vee \zeta_\delta(x_2))(t)=\delta\}$.

Let $f\in C_c^2([0,\infty))$. Then there exists $\delta>0$  such that $f$ is constant on $[0,2\delta]$, and we have
$$
f\left((x_1\vee x_2)(t)\right)=f\left(({\zeta}_{\delta}(x_1)\vee{\zeta}_{\delta}(x_2))(t)\right), \ 0\leq t<+\infty.
$$

We have seen that the local times
$L_0^{\zeta_{\delta}(x_1)-{\zeta}_{\delta}(x_2)}(\cdot)$ and
$L_\delta^{x_1\vee x_2}(\cdot)$ do not increase on $\{t:(x_1\vee
x_2)(t)>2\delta\}$. Taking into account that $f'(x)=f''(x)=0$ on
$[0,2\delta]$ and making use of It$\hat{o}$ formula  we obtain

\begin{multline*}
f((x_1\vee x_2)(t))=f({x_0})+\int_0^t a((x_1\vee x_2)(s))f'((x_1\vee x_2)(s))ds\\+
\int_0^t \sigma((x_1\vee x_2)(s))f'((x_1\vee x_2)(s))dw(s)+
\frac{1}{2}\int_0^t
\sigma^2((x_1\vee x_2)(s))f''((x_1\vee x_2)(s))ds.
\end{multline*}

Therefore, the process $((x_1\vee x_2)(t))_{t\geq0}$ satisfies
equation (\ref{f_martingale}). By the weak uniqueness for all $t\geq0$,  $E^P (x_1\vee x_2)(t)=E^P x_1(t)=E^P x_2(t)$. This yields $(x_1\vee x_2)(t)=x_1(t)=x_2(t), \ t\geq0$ $P$-a.s.
\end{proof}
\begin{proof}[Proof of Theorem] Let for each ${x_0}\geq0$, there exists a solution to the martingale problem $M(a,\sigma,x_0)$.
Then by Lemma \ref{Prop well-posed} the weak uniqueness holds for
equation (\ref{f_martingale}). Then the assertion of Theorem
follows from Lemma \ref{pathwise_uniqueness} similarly to
Yamada-Watanabe theorem (cf. Theorem IV-1.1 of \cite{Ikeda+81}).
\end{proof}
\begin{Remark}
The statement of the Theorem holds true if $\sigma=0$  on some set
$B\subset(0,\infty)$. Indeed, let at first the set $B$ does not
have limit points in some neighborhood of 0. Suppose ${x_0}\in B$
and $a({x_0})=0$. Then a solution of equation (\ref{f_martingale})
stays at the point ${x_0}$ forever. Suppose $a({x_0})\neq0$. If
for some $y\in B$ such that $\ y<{x_0}$, $a(y)=0$, a solution
starting from ${x_0}$ never hits the point $y$ due to homeomorphic
property of solutions of SDE (see \cite{Kunita90}, Ch. 4.4). If
there exists $y\in B$ such that $y\leq{x_0}$ and $a(y)>0$, then a
solution of (\ref{f_martingale}) never attends the half-interval
$[0,y)$. In two last cases the assertion of the Theorem is
fulfilled because $a, \sigma$ are Lipshitz continuous on $[y,\infty)$
(see, for example \cite{Gikhman+68engl}).
 If for all $y\in B$ the inequality  $y>{{x_0}}$ holds, we need to prove the uniqueness only up to the time of hitting B by a solution.  The case when for all $y\in B$ such that $y\leq{{x_0}},$ we have $a(y)<0$ is reduced to the previous one. Thus, if the set $B$ does not have limit points in some neighborhood of $0$, then a solution to equation (\ref{f_martingale}) either attends a point of $B$ just once or does not attends it or lives in it forever. The assertion of Theorem holds true in this case.

Now, suppose that $0$ is a limit point of the set $B$. Suppose
there exists a subsequence $\{y_n:n\geq1\}\subset B$ such that
$y_n\to 0$ as $n\to\infty$ and $a(y_n)\geq0$. Then starting away
from $0$, a solution never hits some neighborhood of $0$ and the
assertion of Theorem holds true. If such a subsequence does not
exist, then there is a neighborhood of $0$, say $U$, such that for
all $y\in B\cap U$, $a(y)<0$. In this case, starting from $0$ a
solution  does not attend  the interval $(y,+\infty)$ for all
$y\in B\cap U$. Thus, if $0$ is a limit point of $B$, a solution
either hits $0$ in finite time with positive probability or does
not hit $0$ a.s.  In the former case a solution stays at the
point $0$ forever. But this contradicts with item $(iii)$ of
Definition \ref{mart_prob}. In the latter case the assertion of
Theorem is obvious.
\end{Remark}

\section{Examples}\label{Examples}
\begin{Example} [Skorokhod equation (\cite{Gikhman+68engl},  \S 23)]
Let $a, b$ be  functions on $[0,\infty)$. Let $(w(t))_{t\geq0}$ be
a Wiener process, ${{x_0}}\geq0$. Recall the definition of a
solution to Skorokhod problem.

Let $(x,l)$ be a pair of continuous processes adapted to the
filtration $(\overline{\mathfrak F}^w_t)$ and such that
\begin{enumerate}[(i)]
    \item $x$ is non-negative,
    \item $l(0)=0, l(\cdot)$ is nondecreasing,
    \item $l(\cdot)$ increases only at those moments of time
        when $x(t)=0$, i.e. for each $t\geq 0$,
        \begin{equation}\label{nosii}
        \int_0^t\mathds{1}_{\{0\}}(x(s))dl(s)=l(t),
        \end{equation}
    \item for each $t\geq0$, the relation
    \begin{equation}\label{Skorokhod}
x(t)={{x_0}}+\int_0^t a(x(s))ds+\int_0^t b(x(s))dw(s)+l(t)
\end{equation}
    holds and all the integrals in the right-hand side of  (\ref{Skorokhod})  are well-defined.

    Then the pair $(x,l)$ is called a strong solution to equation (\ref{Skorokhod}).
\end{enumerate}
If $(x,l)$ is such a solution, then for each $f\in C_c^2([0,\infty))$, by It$\hat{o}$ formula for semimartingales, we have
\begin{multline} \label{Ex.1Eq}
f(x(t))=f({{x_0}})+\int_0^t f'(x(s))b(x(s))dw(s)+\int_0^t f'(x(s))a(x(s))ds\\+
\frac{1}{2}\int_0^t f''(x(s))b^2(x(s))ds+\int_0^t f'(x(s))dl(s).
\end{multline}
According to (\ref{nosii}), the last member in the right-hand side of (\ref{Ex.1Eq}) is equal to 0. Thus, if the pair of the processes $(x,l)$ is a strong solution to equation (\ref{Skorokhod}), then  the process $x$ is a strong solution to equation (\ref{f_martingale}) in the sense of Definition \ref{Def_Str_Sol}.
\end{Example}

\begin{Example} [$\beta$-dimensional Bessel processes] Let $\rho$ be a Bessel process of dimension $\beta$.  It is known ( see \cite{Revuz+99}, p.446) that  this process has a transition probability density
$$
p_t^\beta(x,y)=t^{-1}(y/x)^\nu y  \exp{(-(x^2+y^2)/2t)}I_\nu(xy/t)\ \mbox{for} \  x>0, t>0,
$$
and
$$
p_t^\beta(0,y)=2^{-\nu}t^{-(\nu+1)}\Gamma^{-1}(\nu+1)y^{2\nu+1}\exp(-y^2/2t),
$$
where $\nu=\beta/2-1.$ If $0<\beta<2$ the point $0$ is instantaneously reflecting and for $\beta\geq 2$ it is polar.
\begin{enumerate}[1)]
\item Let $\beta>1$. In this case the process $\rho$ is a semimartingale, which satisfies the SDE of the form (see \cite{Revuz+99}, Ch.XI, \S1).
 \begin{equation}\label{bessel_eq_1}
 \rho(t)=\rho(0)+w(t)+\frac{\beta-1}{2}\int_0^t\frac{1}{\rho(s)}ds, \ \rho(0)\geq0.
 \end{equation}
  Cherny \cite{Cherny00} has shown that there exists a unique
  non-negative strong solution to equation
  (\ref{bessel_eq_1}). Let $\rho$ be a non-negative solution to
  (\ref{bessel_eq_1}). Applying It$\hat{o}$ formula we get the
  equation
  $$
  f(\rho(t))=f(\rho(0))+\int_0^t f'(\rho(s))dw(s)+\frac{\beta-1}{2}\int_0^t\frac{f'(\rho(s))}{\rho(s)}ds+\frac{1}{2}\int_0^t f''(\rho(s))ds.
  $$
  Thus, $(\rho,w)$ is a weak solution to equation (\ref{f_martingale}) in the sense of Definition \ref{Def_Weak_Sol}. Then according to Theorem there exists a strong solution to (\ref{f_martingale}) and the strong uniqueness holds. Therefore we obtain the result of Cherny from ours.
\item Let $0<\beta<1$. Let $(\rho(t))_{t\geq0}$ be a Bessel process on some probability space $(\Omega,\mathfrak{F},(\mathfrak{F}_t),P)$.
   Then the process $\rho$ is not a semimartingale. Nonetheless it has the family of local times defined by the formula
\begin{equation} \label{occupation times}
\int_0^t\phi(\rho(s))ds=\int_0^\infty \phi(x)L_x^\rho(t)x^{\beta-1}dx.
\end{equation}
valid  for all $t>0$ and for every  positive measurable function $\phi$ on $[0,\infty)$ a.s. The Bessel process of dimension $\beta$ is a weak solution to the following equation (cf. \cite{Revuz+99}, Ch.XI, Ex. 1.26)
\begin{equation}\label{bessel_eq}
\rho(t)=\rho(0)+w(t)+\frac{\beta-1}{2}k(t), \ \rho(0)\geq0,
\end{equation}
where $(w(t))_{t\geq0}$ is an $(\mathfrak{F}_t)$-Wiener process, $k(t)=V.P.\int_0^t\rho^{-1}(s)ds$ which, by definition, is equal to $\int_0^\infty a^{\beta-2}(L_a^\rho(t)-L_0^\rho(t))da$.

 Let us check that the pair $(\rho,w)$ is a weak solution to equation (\ref{f_martingale}) in the sense of Definition \ref{Def_Weak_Sol}. Then the  Theorem yields that  there exists a unique strong solution to equation (\ref{bessel_eq}). To prove this we need the following Lemma.
\begin{Lemma}
{\it Let $t_1, t_2\in\mathds{Q},\ t_1<t_2.$ Then for almost all $\omega\in\Omega$ such that $\rho(t,\omega)>0, \
t\in[t_1,t_2],$ the equality
\begin{equation}\label{k-k}
k(t_2)-k(t_1)=\int_{t_1}^{t_2}\frac{1}{\rho(s)}ds
\end{equation}
holds.}
\end{Lemma}
\begin{proof}
There exists $\varepsilon>0$ such that $\rho(t)\geq \varepsilon, \ t\in[t_1,t_2]$. The properties of the local time imply  that for all $a<\varepsilon,$
$L_a^\rho(t_2)=L_a^\rho(t_1).$ Then
\begin{multline*}
k(t_2)-k(t_1)=\int_0^\infty a^{\beta-2}\left[(L_a^\rho(t_2)-L_0^\rho(t_2))-(L_a^\rho(t_1)-L_0^\rho(t_1))\right]da\\=
\int_0^\infty \frac{\mathds{1}_{[\varepsilon,\infty)}(a)}{a}a^{\beta-1}(L_a^\rho(t_2)-L_a^\rho(t_1))da.
\end{multline*}
By (\ref{occupation times}) we get
$$
k(t_2)-k(t_1)=\int_{t_1}^{t_2}\mathds{1}_{[\varepsilon,\infty)}(\rho(s))\frac{1}{\rho(s)}ds=\int_{t_1}^{t_2}\frac{1}{\rho(s)}ds.
$$
\end{proof}
Because of the continuity of the process $(\rho(t))_{t\geq0}$ there
exists $\tilde\Omega\in\Omega, \ P(\tilde\Omega)=1,$ such that
for all $\omega\in\tilde\Omega$, formula (\ref{k-k}) holds true
for all $t_1,t_2\geq 0$ satisfying $\rho(t)>0, t\in[t_1,t_2]$.

Let $\tau$ be a stopping
time such that $\rho(\tau)\neq 0$ P-a.s. Put $\sigma=\inf\{s\geq\tau:\
\rho(s)=0\}.$ We have
$$
\rho(t)=\rho(\tau)+w(t)-w(\tau)+\frac{\beta-1}{2}\int_\tau^t\frac{ds}{\rho(s)},\ t\in[\tau,\sigma).
$$

Let $f\in C_c^2([0,\infty)).$ It$\hat{o}$  formula for semimartingales yields
\begin{equation}\label{frt}
f(\rho(t))-f(\rho(\tau))=\int_\tau^t f'(\rho(s))dw(s)+\frac{\beta-1}{2}\int_\tau^t \frac{f'(\rho(s))}{\rho(s)}ds+\frac{1}{2}\int_\tau^t f''(\rho(s))ds, \ t\in[\tau,\sigma).
\end{equation}

Choose  $\delta>0$ such that   $f$ is constant on $[0,2\delta].$ Define
\begin{eqnarray*}
\tau_0&=&0,\\
\mbox{for}\ i\geq0,\ \varkappa_i&=&\inf\{t>\tau_i:\rho(t)=\delta/2\}, \\
\mbox {and for}\ i\geq 1,\
\tau_i&=&\inf\{t>\varkappa_{i-1}:\rho(t)=\delta\},\\
\end{eqnarray*}
Then
$$
f(\rho(t))=f(\rho(0))+\sum_{k=0}^\infty[f(\rho(\varkappa_k\wedge t))-f(\rho(\tau_k\wedge t))]+\sum_{k=0}^\infty[f(\rho(\tau_{k+1}\wedge t))-f(\rho(\varkappa_k\wedge t))].
$$
The second sum in the right-hand side is equal to zero. If
$f(\rho(0))<\delta/2$, then \break $f(\rho(\varkappa_0\wedge
t))-f(\rho(\tau_0))=0$.

Suppose $\rho(0)\geq\delta/2$. It follows from (\ref{frt}) that
\begin{multline}\label{frt2}
f(\rho(t))=f(\rho(0))+\sum_{k=0}^\infty\int_{\tau_k\wedge t}^{\varkappa_k\wedge t} f'(\rho(s))dw(s)+\frac{\beta-1}{2}\sum_{k=0}^\infty\int_{\tau_k\wedge t}^{\varkappa_k\wedge t} \frac{f'(\rho(s))}{\rho(s)}ds\\+
\frac{1}{2}\sum_{k=0}^\infty\int_{\tau_k\wedge t}^{\varkappa_k\wedge t} f''(\rho(s))ds, \ t\geq 0.
\end{multline}
Note that  for all $t\in[\varkappa_k,\tau_{k+1}], \ k\geq0$, $f'(\rho(t))=f''(\rho(t))=0$. Then  (\ref{frt2}) can be rewritten in the form
\begin{equation}\label{frt3}
f(\rho(t))=f(\rho(0))+\int_0^t f'(\rho(s))dw(s)+\frac{\beta-1}{2}\int_0^t \frac{f'(\rho(s))}{\rho(s)}ds+\frac{1}{2}\int_0^t f''(\rho(s))ds, \ t\geq 0.
\end{equation}
If $\rho(0)<\delta/2$ equation (\ref{frt3}) can be obtained similarly.

Hence, the pair $(\rho,w)$ is a weak solution to equation
(\ref{f_martingale}) and, consequently, the strong existence
and uniqueness hold for equation (\ref{bessel_eq}).
\end{enumerate}
\end{Example}

\begin{Example}
Let the process $(x(t))_{t\geq0}$ be a weak solution to an SDE of the form
\begin{equation}\label{Girsanov1}
x(t)=x(0)+\int_0^t a(|x(s)|)ds+\int_0^t b(|x(s)|)dw(s),
\end{equation}
where the coefficients $a$ and $b$ are locally Lipshitz continuous
on $(0,\infty)$.

Then for each even function being constant in a neighborhood of
zero according to It$\hat{o}$ formula, we get
\begin{multline}\label{Girsanov}
f(x(t))=f(x(0))+\int_0^t a(|x(s)|)f'(x(s))ds+\int_0^t
b(|x(s)|)f'(x(s))dw(s)\\
+\frac{1}{2}\int_0^t b^2(|x(s)|)f''(x(s))ds.
\end{multline}
Note that if the process $(x(t))_{t\geq0}$ is a weak solution to
equation (\ref{Girsanov1}) spending zero time at the origin  then
the process $y(t)=|x(t)|, \ t\geq0,$ satisfies equality
(\ref{Girsanov}) for each  $f\in C_c^2([0,+\infty))$. By Theorem 1
the process $y(t),\ t\geq 0,$ is a unique non-negative strong
solution  to (\ref{Girsanov}) spending zero time at the origin.

Consider an SDE which can be regarded as an example of equation of
the form (\ref{Girsanov1})
\begin{equation}\label{Girsanov_equation}
x(t)=x(0)+\int_0^t |x(s)|^\alpha dw(s), \ \alpha\in(0,1/2).
\end{equation}
It is known that there exists a weak solution to (\ref{Girsanov_equation})
spending zero time at the point $0$  (cf. \cite{McKean69}, 3.10b).
\begin{Remark}
Girsanov \cite{Girsanov62} has shown that without additional assumption this equation has infinitely many weak solutions.
\end{Remark}
\begin{Remark}
It can be proved (cf. \cite{Bass+07}) that in the class of solutions spending zero time at the point $0$ the pathwise uniqueness holds and a strong solution exists.
\end{Remark}

So, there exists a weak solution to the equation
\begin{equation}\label{Girsanov_equation2}
f(x(t))=f(x(0))+\int_0^t(x(s))^\alpha f'(x(s))dw(s)+\frac{1}{2}\int_0^t(x(t))^{2\alpha}f''(x(s))ds
\end{equation}
in the sense of Definition 3 spending zero time at the point $0$. According to Theorem 1 there is a strong solution to (\ref{Girsanov_equation2}) spending zero time at the point $0$. Certainly, this solution coincides with the unique strong solution to the equation
$$
x(t)=x(0)+\int_0^t(x(s))^\alpha dw(s)+ dL_{0}^x(t), \ \alpha\in(0,1/2),
$$
spending zero time at the point $0$ which was constructed by Bass and Chen (see \cite{Bass+05}). Here $(L_{0}^x(t))_{t\geq0}$ is a local time of the process $(x(t))_{t\geq0}$ at the point 0 defined by formula (\ref{local time}).
\end{Example}

\section*{Appendix}
\begin{proof}[Proof of assertion of Remark 2]
The "only if" assertion is trivial.

To prove the "if" assertion we can argue as in Prop.2.1, Ch.IV of
\cite{Ikeda+81}. Suppose $P$ is a solution to the martingale
problem $M(a,\sigma,{{x_0}})$ on space
$(C^+([0,+\infty)),\mathfrak{G}),(\mathfrak{G}_t))$, and $f\in
C_c^2([0,\infty))$. Then the process $Y_f(t)$ is a continuous,
square integrable local martingale with respect to $P$. Applying
condition (ii) of Definition \ref{mart_prob} to the function
$f^2$, we calculate the characteristics of the process
$(Y_f(t))_{t\geq0}$. Namely,
$$
\langle Y_f\rangle(t)=\int_0^t\sigma^2(x(s))(f'(x(s)))^2ds.
$$
Consequently,
there is a Brownian motion $(w_f(t))_{t\geq0}$ defined on an
extension  of \break $(C^+([0,\infty)),\mathfrak{G},(\mathfrak{G}_t),P)$ such that
$$
Y_f(t)=\int_0^t\sigma(x(s))f'(x(s))dw_f(s).
$$
We will show that it can be chosen the same Brownian motion for
all $f\in C_c^2([0,\infty))$.

Similarly to the Proof of Lemma 1, for $k=1,2,\dots$, consider a non-decreasing function $\eta_k\in
C_c^2([0,\infty))$
 such that $\eta_k(x)=x, \ x>1/k,$ and $\eta_k$ is a constant on $[0,\frac{1}{2k}]$.

Let us fix $k$ and put
$$
\tau_l=\inf\{t:\eta_k(x(t))>l\}, \ l=1,2,\dots.
$$
Then, for all $l=1,2,\dots,$
$$
\eta_k(x(t\wedge\tau_l))-\eta_k({{x_0}})-\int_0^{t\wedge\tau_l}\left[a(x(s))\eta_k'(x(s))+\frac{1}{2}\sigma^2(x(s))\eta_k''(x(s))\right]ds
$$
is a continuous, square integrable $P$-martingale. Then
$Y_{\eta_k}(t)\in\mathcal{M}_2^{c,loc}(P)$, and there exists a
Brownian motion $(w_k(t))_{t\geq0}$ on an extension $(\Omega_k,
\mathfrak{F}_k, P_k )$ of
$(C^+([0,\infty)),\mathfrak{G},(\mathfrak{G}_t),P)$ such that
\begin{multline} \label{eta_martingale}
    \eta_k(x(t))=\eta_k(x_0)+\int_0^t\left(a(x(s))\eta_k'(x(s))+\frac{1}{2}\sigma^2(x(s))\eta_k''(x(s))\right)ds\\+
    \int_0^t \sigma(x(s))\eta_k'(x(s))dw_k(s).
\end{multline}
Fix $m\geq1$. Then for all $k\geq m$, $\eta_m(x)=\eta_k(x),\ x>1/m.$ Put
\begin{equation}\label{tilde_w_m}
\tilde{w}_m(t):=\int_0^t\mathds{1}_{(\frac{1}{m},+\infty)}(x(s))dw_m(s).
\end{equation}
As a consequence of the following simple Lemma we have that for each $m\geq1$ the process $(\tilde{w}_m(t))_{t\geq0}$ is adapted w.r.t. the filtration generated by the process $(x(t))_{t\geq0}$ and for all $k\geq m$,
\begin{equation}\label{w_m_equality}
\int_0^t\mathds{1}_{(\frac{1}{m},+\infty)}(x(s))d w_k(s)=\int_0^t\mathds{1}_{(\frac{1}{m},+\infty)}(x(s))dw_m(s)=\int_0^t\mathds{1}_{(\frac{1}{m},+\infty)}(x(s))d \tilde w_m(s) \ \mbox{ a.s.}
\end{equation}

\begin{Lemma}\label{Wiener_equality} {\it Let $A$ be an open set in $\mathds{R}$.
Let $ x_0\in A$, $(x(t))_{t\geq0}$ be a continuous
adapted process on a probability space $(\Omega,\mathcal{F}_t,P)$.
Let $(w(t))_{t\geq0}$ be a Wiener process on some extension of the
space $(\Omega,\mathcal{F}_t,P)$. Suppose $a,b,f$ are continuous
functions on $\mathbb{R}$, $b(x)\neq0$ for $x\in A$, and for all
$t\geq0$, the equality
$$
f(x(t))=f(x_0)+\int_0^t a(x(s))ds+\int_0^t b(x(s))dw(s)
$$
holds. Put $\mathcal{F}_t^x=\sigma\{x(s):0\leq s\leq t\}$. Then
the process $\int_0^t\mathds{1}_{A}(x(s))dw(s),
t\geq 0$, is adapted w.r.t. $(\mathcal{F}_t^x)$.

Moreover, suppose $(\bar{w}(t))_{t\geq0}$  is a Wiener process on
an extension of the probability space $(\Omega,\mathcal{F}_t,P)$,
$\bar{a},\bar{b},\bar{f}$ are continuous functions on $\Bbb R$,
$\bar{b}(x)\neq0$ for $x\in A$, and the equality
$$
\bar f(x(t))=\bar f(x_0)+\int_0^t \bar a(x(s))ds+\int_0^t \bar
b(x(s))d\bar w(s)
$$
holds.

If  $a(x)=\bar{a}(x), b(x)=\bar{b}(x), f(x)=\bar{f}(x)$ on
$A$, then
\begin{equation}\label{equality}
\int_0^t\mathds{1}_{A}(x(s))dw(s)=\int_0^t\mathds{1}_{A}(x(s))d\bar w(s).
\end{equation}}
\end{Lemma}
The proof is trivial.

The sequence $\{\tilde{w}_m:\ m\geq1\}$ defined in (\ref{tilde_w_m}) is fundamental in mean square on compact intervals. Indeed, for $k\geq m, T>0$, using martingale inequality (cf. \cite{Ikeda+81}, Theorem I-6.10) and (\ref{w_m_equality}), we get
$$
E\left[\sup_{t\in[0,T]}|\tilde w_k(t)-\tilde w_m(t)|\right]^2\leq 4E|\tilde w_k(T)-\tilde w_m(T)|^2=
$$
$$
4E\left[\int_0^T\left(\mathds{1}_{(\frac{1}{k}, +\infty)}-\mathds{1}_{(\frac{1}{m}, +\infty)}\right)(x(s))dw_k(s)\right]^2\leq 4E\int_0^T \mathds{1}_{(\frac{1}{k},\frac{1}{m}]}(x(s))ds\to 0 , m\to +\infty.
$$
Then the sequence $\{\tilde w_m: \ m\geq1\}$ is uniformly
convergent on compact intervals in probability. Denote the limit
of the sequence $\{\tilde w_m: \ m\geq1\}$ by $\tilde w.$

The process $(\tilde{w}(t))_{t\geq0}\in\mathcal{M}_2^{c,loc}(P)$ and
$$
\langle\tilde{w}(t)\rangle(t)=\int_0^t\mathds{1}_{(0,\infty)}(x(s))ds=t.
$$
Here we used the fact that the process $(x(t))_{t\geq0}$ spends
zero time at the point $0$. Thus the process
$(\tilde{w}(t))_{t\geq0}$ is a Wiener process. Besides, by
construction,
$$
\tilde{w}_k(t)=\int_0^t\mathds{1}_{(\frac{1}{k},+\infty)}(x(s))d\tilde{w}(s).
$$
Let $f\in C_c^2([0,\infty))$ be such that $f$ is constant on $[0,1/k]$. Then there exists a Wiener process $(w_f(t))_{t\geq0}$ such that
\begin{multline}\label{f_Wiener}
   f(x(t))=f(x_0)+\int_0^t\left(a(x(s))f'(x(s))+\frac{1}{2}\sigma^2(x(s))f''(x(s))\right)ds
   \\+
    \int_0^t \sigma(x(s))f'(x(s))d{w_f}(s).
\end{multline}
By It$\hat{o}$ formula,  (\ref{eta_martingale}) yields
\begin{multline}\label{f_eta_k}
   f(\eta_k(x(t)))=f(\eta_k(x_0))+\int_0^t a(x(s))\eta_k'(x(s))f'(\eta_k(x(s)))ds\\+
   \frac{1}{2}\int_0^t\sigma^2(x(s))\eta_k''(x(s))f'(\eta_k(x(s)))ds\\+
   \int_0^t \sigma(x(s))\eta_k'(x(s))f'(\eta_k(x(s)))d{w_k}(s)+\frac{1}{2}\int_0^t\sigma^2(x(s))(\eta_k'(x(s)))^2f''(\eta_k(x(s)))ds.
\end{multline}
The second integral in the right-hand side  of (\ref{f_eta_k}) is equal to $0$ because $\eta_k''(x)=0$ on $(1/k,+\infty)$. Taking into account that $\eta_k'(x)=x'=1$ on $(1/k,+\infty)$, and $f(\eta_k(x))=f(x)$ on $(1/k,+\infty)$, we arrive at the equation
\begin{multline}\label{f_eta_k_2}
f(\eta_k(x(t)))=f(\eta_k(x_0))+\int_0^t a(x(s))f'(x(s))ds+\int_0^t\sigma(x(s))f'(x(s))dw_k(s)\\+\frac{1}{2}\int_0^t \sigma^2(x(s))f''(x(s))ds.
\end{multline}
Note that
$$
\int_0^t\sigma(x(s))f'(x(s))dw_k(s)=\int_0^t\sigma(x(s))f'(x(s))\mathds{1}_{\{\sigma(x(s))f'(x(s))\neq0\}}dw_k(s).
$$
Applying Lemma \ref{Wiener_equality} to  equations (\ref{f_Wiener}) and (\ref{f_eta_k_2}) we have
\begin{multline*}
\int_0^t\mathds{1}_{(1/k,+\infty)}(x(s))\mathds{1}_{\{\sigma(x(s))f'(x(s))\neq0\}}dw_f(s)\\=\int_0^t\mathds{1}_{(1/k,+\infty)}(x(s))\mathds{1}_{\{\sigma(x(s))f'(x(s))\neq0\}}dw_k(s)\\
=\int_0^t\mathds{1}_{(1/k,+\infty)}(x(s))\mathds{1}_{\{\sigma(x(s))f'(x(s))\neq0\}}d\tilde w(s).
\end{multline*}
So for each $f\in C_c^2([0,+\infty))$ the equality
\begin{multline*}
   f(x(t))=f(x_0)+\int_0^t\left(a(x(s))f'(x(s))+\frac{1}{2}\sigma^2(x(s))f''(x(s))\right)ds\\+
    \int_0^t \sigma(x(s))f'(x(s))d\tilde w(s)
\end{multline*}
is justified, and the pair
$(x(t),\tilde w(t))_{t\geq0}$ is a weak solution to equation
(\ref{f_martingale}).
\end{proof}



\end{document}